\documentclass[12pt]{article}

\RequirePackage{kpfonts}
\RequirePackage[sf,bf,small,raggedright]{titlesec}
\usepackage{dsfont}
\usepackage{comment}
\usepackage{textcomp}
\usepackage{amssymb}
\usepackage{bbm}
\usepackage{fancyhdr}
\usepackage{amsmath}
\usepackage{amsbsy,amsthm}
\usepackage{amscd}
\usepackage{latexsym}
\usepackage{graphicx}   
\usepackage{pdfsync}
\usepackage{blkarray}
\usepackage{multirow}
\usepackage{hyperref}
\usepackage{xcolor}
\usepackage{enumerate}

\usepackage{tcolorbox}
\usepackage{soul}

\usepackage{tikz}

\usepackage{subcaption}
\usepackage{algorithmic}
\usepackage[justification=centering]{caption}

\usepackage{natbib}
 \def\newblock{\ }%
 \bibpunct[, ]{[}{]}{,}{n}{}{,}%

\newcommand{\N}{\mathbb{N}}
\newcommand{\Z}{\mathbb{Z}}

\newtheorem{lemma}{Lemma}
\newtheorem{theorem}{Theorem}

\newtheorem{proposition}{Proposition}

\allowdisplaybreaks
\newtheorem{remark}{Remark}

\numberwithin{equation}{section}

\def\IND{\mathbbm{1}}

\newcommand{\EXP}{\mathbb{E}}
\newcommand{\PROB}{\mathbb{P}}

\newcommand{\var}{\mathrm{Var}}

\newcommand{\G}{{\mathcal G}}

\newcommand{\defeq}{\stackrel{\mathrm{def.}}{=}}

\begin{document}

\title{{Subtractive random forests
    \thanks{
Luc Devroye is supported by the Natural Sciences and Engineering Research Council of Canada
(NSERC)
  G. Lugosi acknowledges the support of Ayudas Fundación BBVA a
Proyectos de Investigación Científica 2021 and
the
Spanish Ministry of Economy and Competitiveness grant PID2022-138268NB-I00, financed by MCIN/AEI/10.13039/501100011033,
FSE+MTM2015-67304-P, and FEDER, EU.
}}
\author{
  Nicolas Broutin
\thanks{LPSM, Sorbonne Universit\'e, 4 Place Jussieu, 75005 Paris}
\thanks{Institut Universitaire de France (IUF)}
\and
  Luc Devroye
\thanks{School of Computer Science, McGill University, Montreal, Canada}
\and  
G\'abor Lugosi
\thanks{Department of Economics and Business, Pompeu
  Fabra University, Barcelona, Spain}
\thanks{ICREA, Pg. Lluís Companys 23, 08010 Barcelona, Spain}
\thanks{Barcelona Graduate School of Economics}
\and
Roberto Imbuzeiro Oliveira
\thanks{IMPA, Rio de Janeiro, RJ, Brazil}
}
}

\maketitle

\begin{abstract}
Motivated by online recommendation systems, we study a family of
random forests.  The vertices of the forest are labeled by integers.
Each non-positive integer $i\le 0$ is the root of a tree. Vertices
labeled by positive integers $n \ge 1$ are attached sequentially such
that the parent of vertex $n$ is $n-Z_n$, where the $Z_n$ are i.i.d.\
random variables taking values in $\N$.
We study several characteristics of the resulting random forest. In
particular, we establish
bounds for the expected tree sizes, the number of trees in the
forest, the number of leaves, the maximum degree, and the height of
the forest. 
We show that for all distributions of the $Z_n$, the forest
contains at most one infinite tree, almost surely.
If $\EXP Z_n < \infty$, then there
is a unique infinite tree and the total size of the remaining trees is
finite, with finite expected value if  $\EXP Z_n^2 <
\infty$.
If $\EXP Z_n = \infty$ then almost surely all trees are finite. 
\end{abstract}


\section{Introduction}

In some online recommendation systems a user receives recommendations
of certain topics that are selected sequentially, based on the past
interest of the user. At each time instance, the system chooses a
 topic by selecting a random time length, subtracts this length from
 the current date and recommends the same
 topic that was recommended in the past at that time.  Initially there is an
 infinite pool of topics. The random time lengths are assumed to be
 independent and identically distributed.

 The goal of this paper is to study the long-term behavior of such
 recommendation systems. We suggest a model for such a system
 that allows us to understand many of the most important properties.
 For example, we show that if the expected subtracted time length has
 finite expectation, then, after a random time, the system will
 recommend the same topic forever. When the expectation is infinite,
 all topics are recommended only a finite number of times. 

The system is best understood by studying properties of random
forests that we coin \emph{subtractive random forest} (SuRF).
Every tree in the forest corresponds to a topic and vertices are
attached sequentially, following a subtractive attachment rule.

To define the mathematical model, 
we consider sequential
random coloring of the positive integers as follows.
Let $Z_1,Z_2,\ldots$ be independent, identically distributed random variables, taking values
in the set of positive integers $\N$. Define $C_i=i$ for all nonpositive integers
$i\in \{0,-1,-2,\ldots\}$.  
We assign colors to 
the positive integers $n\in \N= \{1,2,\ldots\}$ by the recursion
\[
     C_n = C_{n-Z_n}~.
\]
This process naturally defines a random forest whose vertex set is $\Z$.
Each $i\in \{0,-1,-2,\ldots\}$ is the root of a tree in the forest.
The tree rooted at $i$ consists of the vertices corresponding to all  $n\in \N$ such that $C_n=i$.
Moreover, there is an edge between vertices $n'<n$ if and only if $n'=n-Z_n$.
Figure \ref{fig:srf}.

In other words, trees of the forest are obtained by sequentially
attaching vertices corresponding to the positive
integers. Denote the tree rooted at $i\in \{0,-1,-2,\ldots\}$ at time
$n$ by $T^{(i)}_n$ (i.e., the tree rooted at $i$ containing vertices
with index at most $n$).
Initially, all trees of the forest contain a single vertex: $T^{(i)}_0=\{i\}$.
At time $n$, vertex $n$ is added to the tree rooted at $i=C_n$ such that
$n$ attaches by an edge to vertex $n-Z_n$. All other trees remain unchanged,
that is, $T^{(i)}_n=T^{(i)}_{n-1}$ for all $i \neq C_n$.

Define $T^{(i)}= \cup_{n\in \N} T^{(i)}_n$ as the random (possibly infinite) tree
rooted at $i$ obtained at the ``end'' of the random attachment process. 

We study the behavior of the resulting forest. The following random variables are of
particular interest :
\begin{eqnarray*}
  S_n^{(i)} &=& \sum_{t=1}^n \IND_{C_t=i}   \quad \text{(the size of tree $T^{(i)}_n$ excluding the root $i\le 0$)~;} \\
  S^{(i)} &=& \sum_{t \in \N} \IND_{C_t=i}  \quad \text{(the size of tree $T^{(i)}$ excluding the root $i\le 0$)~;} \\
  M_n &=& \sum_{i \in \{0,-1,-2,\ldots\}} \IND_{S^{(i)}_n >1}~;  \\
            & &\quad \text{(the number of trees with at least one vertex attached to the root at time $n$ )} \\
  D_n^{(i)} & = & \sum_{t=1}^n \IND_{Z_t = t-i} \quad \text{(the
                  degree of vertex $i$ at time $n$)~.} \\
\end{eqnarray*}
Introduce the notation
\[
   q_n = \PROB\{Z=n\} \quad \text{and} \quad p_n = \PROB\{Z\ge n\} =
   \sum_{t \ge n} q_t
\]
for $n\in \N$, where $Z$  is a random variable distributed as the
$Z_n$.

Another key characteristic of the distribution of $Z$ is 
\[
m_n \defeq \sum_{t=1}^n p_t = \EXP \left[ \min(Z,n) \right]~.
\]
Note that $m_n$ is nondecreasing in $n$ and is bounded if and
only if $\EXP Z < \infty$.

Finally, let $r_n= \PROB\{C_n=0\}$ denote the probability that vertex $n$ belongs to the tree rooted at $0$.
Then $r_n$ satisfies the recursion
\begin{equation}
\label{eq:basicrecursion}
    r_n = \sum_{t=1}^n q_t r_{n-t}~,
\end{equation}
for $n\in \N$, with $r_0=1$.

\begin{figure}[t]
\begin{center}
\leavevmode
\includegraphics[width=\textwidth]{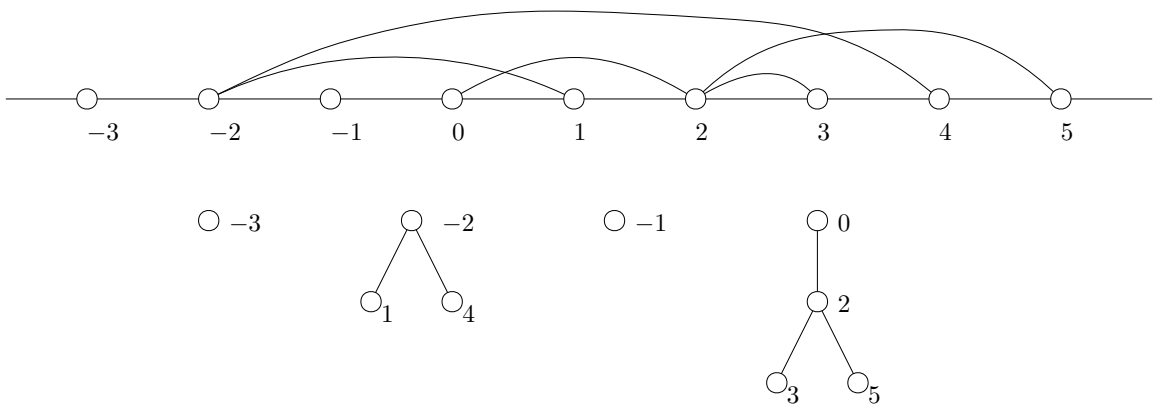}
\end{center}
\caption{An example of the subtractive attachment process (up to time
  $5$) and the
  resulting forest. Here $Z_1=3$, $Z_2=2$, $Z_3=1$, $Z_4=6$, and $Z_5=3$.}
\label{fig:srf}
\end{figure}

The paper is organized as follows. In Section \ref{sec:survival} we
study whether the trees $T^{(i)}$ of the forest are finite or infinite. We show
that it is the finiteness of the expectation of $Z$ that characterizes the behavior of
the forest in this respect. 
In particular, if $\EXP Z = \infty$, then,
almost surely,
all trees of the forest are finite. On the other hand, if
$\EXP Z < \infty$, then the forest has a unique infinite tree and the
total number of non-root vertices in finite trees is finite almost surely.

In Section \ref{sec:meantreesize} the expected size of the trees $T_n^{(i)}$
is studied at time $n$. It is shown that when $Z$ has full
support, the expected size of
each tree at time $n$ of the forest tends to infinity, regardless
of the distribution. The expected tree sizes are sublinear in $n$ if and only if $\EXP Z = \infty$.

We also study various parameters of the random trees of the forest. In
Sections \ref{sec:leaves}, \ref{sec:degrees}, and \ref{sec:height}
we derive results on the number of leaves, vertex degrees, and the height
of the forest, respectively.

\subsection{Related work}

The random subtractive process studied here was examined
independently, in quite different contexts, 
by
Hammond and Sheffield \cite{HaSh13}
and
Baccelli and Sodre \cite{BaSo19},
Baccelli, Haji-Mirsadeghi, and Khezeli \cite{BaHaKh18},
and
Baccelli, Haji-Mirsadeghi, and Khaniha \cite{BaHaKh22}.
These papers consider an extension of the 
 ancestral lineage process to the set $\Z$ of integers defined as follows: let
 $\{Z_n\}_{n\in \Z}$ be i.i.d.\ random variables taking positive
 integer values. This naturally defines a random graph $\G$ with vertex set $\Z$
 such that vertices $m,n\in \Z$ with $m<n$ are connected by an edge
 if and only if $n-Z_n=m$.
If we define the graph $\G^{(0)}$ as the
subgraph of $\G$ obtained by removing all edges $(m,n)$ for which
$\max(m,n)\le 0$, then $\G^{(0)}$ is exactly the subtractive
random forest studied in this paper. 

It is shown in \cite{HaSh13} --- and also in \cite{BaHaKh22} ---
that if $\sum_{n=1}^\infty r_n^2 < \infty$, then almost surely $\G$ has a unique
connected component, whereas  if $\sum_{n=1}^\infty r_n^2 = \infty$, then
almost surely $\G$ has infinitely many
connected components. Hammond and Sheffield  are only interested in the latter (extremely heavy-tailed) case.
They use the resulting coloring of the integers to define a random walk that converges to fractional Brownian motion.
See also Igelbrink and Wakolbinger \cite{IgWa22} for further results
on the urn model of Hammond and Sheffield.

The paper of Baccelli and Sodre \cite{BaSo19} considers the case when
$\EXP Z < \infty$. They show that in this case the graph $\G$ has a
unique doubly infinite path. This implies that the
subtractive random forest $\G^{(0)}$ contains a unique infinite tree. This fact is
also implied by Theorem \ref{thm:finitemean} below. The fact that
all trees of the forest $\G^{(0)}$ become extinct when $\EXP Z = \infty$ (part (ii)
of our Theorem \ref{thm:extinction}) is implicit in Proposition 4.8 of
\cite{BaHaKh22}. In that paper, the graph $\G$ is referred to as
a \emph{Renewal Eternal Family Forest} (or \emph{Renewal Eternal
  Family Tree} when it has a single connected component).

The long-range seed bank model of Blath, Jochen, Gonz{\'a}lez, Kurt, and Spano \cite{BlGoKuSp13} is also 
based on a similar subtractive process.

Another closely related model is studied by Chierichetti, Kumar, and Tomkins \cite{ChKuTo20}. In their model, the
nonnegative integers are colored by a finite number of colors and the subtractive process is not stationary. Given a sequence of 
positive ``weights'' $\{w_i\}_{i=1}^\infty$, the colors are assigned to the positive integers sequentially such that the color
of $n$ is the same as the color of $n-Z_n$ where the distribution of $Z_n$ is given by $\PROB\{Z_n=i\}= w_i/\sum_{j=1}^i w_j$.
(The process is initialized by assigning fixed colors to the first few positive integers.) 
Chierichetti, Kumar, and Tomkins are mostly interested in the existence of the limiting empirical distribution of the colors. 


\section{Survival and extinction}
\label{sec:survival}

This section is dedicated to the question whether the trees $T^{(i)}$ of the
subtractive random forest are finite or infinite. The main results
show a sharp contrast in the behavior of the limiting random forest
depending on the tail behavior of the random variable $Z$. 
When $Z$ has a light tail such that $\EXP Z < \infty$, then a single
tree survives and the total number of non-root vertices of all the remaining trees is finite, almost
surely.
This is in sharp contrast to what happens when $Z$ is heavy-tailed:
When $\EXP Z = \infty$, then all trees become extinct, that is, every
tree in the forest is finite. 




\subsection{$\EXP Z<\infty$: a single infinite tree}
\label{sec:finiteexp}

First we consider the case when $\EXP Z < \infty$. 
We show that, almost surely, the forest contains a
single infinite tree.
Moreover, the
total number of non-root vertices in all other trees is an almost surely finite random variable.
In other words, the sequence of ``colors'' $C_n$ becomes constant
after a random index. 

\begin{theorem}
\label{thm:finitemean}
Let $\EXP Z < \infty$ and assume that $q_1>0$.
Then there exists a positive random variable $N$
with $\PROB\{N<\infty\}=1$ and a (random) index $I\in
\{0,-1,\ldots\}$ such that $C_n=I$ for all $n\ge N$,  
with probability one.
\end{theorem}

\begin{proof}
Define a
Markov chain $\{B_n\}_{n\in \N}$ with state space $\N$ by
the recursion $B_1=1$ and, for $n\ge 1$,
\[
  B_{n+1} = \left\{ \begin{array}{ll}
                      B_n +1 & \text{if $Z_{n+1} \le B_n$}  \\
                      1 & \text{otherwise.}
                          \end{array}  \right.
\]
Thus, $B_n$ defines the length of a block of consecutive vertices such that each vertex in the
block (apart from the first one) is linked to a previous vertex in the
same block. In particular, all vertices in the block belong to the
same tree.

Note first that for any $n >1$,
\begin{eqnarray*}
   \PROB\{ B_n=n \}& = & q_1(q_1+q_2)\cdots (q_1+\cdots + q_{n-1}) \\
                    & = & \prod_{i=1}^{n-1} (1-p_{i+1}) \\
                    & \ge & \exp\left( - \sum_{i=1}^{n-1} \frac{p_{i+1}}{1-p_{i+1}}\right) \\
                    & \ge & \exp\left( - \frac{1}{1-p_2}\sum_{i=1}^{n-1} p_{i+1}\right) \\
   & \ge & \exp\left( - \frac{\EXP Z}{q_1}\right)~.
\end{eqnarray*}
Since the events $\{ B_n=n \}$ are nested, by continuity of measure we
have that
\begin{equation}
\label{eq:escape}
   \PROB\{ B_n=n \ \text{for all $n\in \N$} \} \ge \exp\left( -
     \frac{\EXP Z}{q_1}\right) >0~.
\end{equation}
Hence, with positive probability, $C_n=0$ for all $n\ge 1$.
(Note that $\PROB\{C_1=0\}=q_1$ is positive by assumption.)
Since
 $\{B_n\}_{n\in \N}$ is a Markov chain, this implies that,
 with probability one, the set $\{n: B_n=1\}$ is finite, which implies
 the theorem.
 We may take $N= \max\{n\in \N: B_n=1\}$ as the (random) index after which
the sequence $C_n$ is a constant.
\end{proof}

Note that the assumption $q_1>0$ may be somewhat weakened. However,
some condition is necessary to avoid periodicity. For example, if the distribution
of $Z$ is concentrated on the set of even integers, then the assertion of Theorem \ref{thm:finitemean}
cannot hold.

The next result shows that the random index $N$ has a finite
expectation if and only if $Z$ has a finite second moment.

\begin{theorem}
\label{thm:secondmoment}
Let $\EXP Z < \infty$ and assume that $q_1>0$. Consider random index
$N$ defined in the proof of Theorem \ref{thm:finitemean}.
Then $\EXP N < \infty$ if and only if $\EXP Z^2 < \infty$.
In particular, if $\EXP Z^2 < \infty$, then the total number of
vertices $n\in \N$ outside of the unique infinite tree has finite expectation.
\end{theorem}

\begin{proof}
Consider the Markov chain $\{B_n\}_{n\in \N}$ defined in the proof of Theorem
\ref{thm:finitemean}. For $i\in \N$, let $N_i=\sum_{n=1}^\infty
\IND_{B_n=i}$ denote the number of times the Markov chain visits state
$i$. The key observation is that we may write
\[
    N= \sum_{n=1}^\infty \sum_{i=1}^\infty i \IND_{B_n=i, B_{n+1}=1}~.
\]
Since $\PROB\{B_{n+1}=1 |B_n=i\} =p_{i+1}$, this implies 
\[
    \EXP N= \sum_{i=1}^\infty i p_{i+1}\EXP N_i~.
\]
Next, notice that $\EXP N_2 =q_1\EXP N_1$, $\EXP N_3
=(q_1+q_2)\EXP N_2$, and similarly,
$\EXP N_i =(1-p_i)\EXP N_{i-1}=\prod_{j=2}^i(1-p_j) \EXP N_1$.
{\color{blue}
 By convention, we write $\prod_{j=2}^1(1-p_j)=1$.
}
It follows from \eqref{eq:escape} that $N_1$ is stochastically dominated by a geometric 
random variable and therefore $\EXP N_1 < \infty$.
Thus, 
\[
    \EXP N= \EXP N_1 \sum_{i=1}^\infty i p_{i+1} \prod_{j=2}^i(1-p_j)~.
  \]
 As noted in the proof of Theorem \ref{thm:finitemean}, for all
  $i\ge 1$,
\[
  c_0  \le \prod_{j=2}^i(1-p_j)
    \le 1~,
  \]
where $c_0 \defeq \exp\left( - \frac{\EXP Z}{q_1}\right)$ is a
positive constant,
and therefore
\[
c_0 \EXP N_1 \sum_{i=1}^\infty i p_{i+1} \le  \EXP N \le  \EXP N_1 \sum_{i=1}^\infty i p_{i+1}~.
\]
Since $\sum_{i=1}^\infty i p_{i+1}= (1/2) \EXP Z(Z-1)$, the theorem follows.
\end{proof}

 \subsection{$\EXP Z=\infty$: extinction of all trees}
 \label{sec:infiniteexp}

 In this section we show that when $Z$ has infinite expectation, then every tree of the forest becomes
extinct, almost surely.  In other words, with probability one, there is
no infinite tree in the random forest. This is in sharp contrast with
the case when $\EXP Z<\infty$, studied in Section \ref{sec:finiteexp}.

Recall that for $i\le 0$, $S^{(i)}$ denotes the size of tree $T^{(i)}$
rooted at vertex $i$.

A set of vertices $\{n_1,n_2,\ldots  \} \subset \{0,1,2,\ldots\}$
forms a \emph{maximal infinite path}
if $n_1<n_2< \cdots$, for all $k >1$, $n_k-Z_{n_k}= n_{k-1}$,
and $n_1-Z_{n_1} \le 0$.

\begin{theorem}
\label{thm:extinction}
\begin{itemize}
\item[(i)]
If $\EXP Z < \infty$, then, with probability one, there exists a
unique integer $i \le 0$ such that $S^{(i)}=\infty$ and
the forest contains a unique maximal infinite path. Moreover,
\[
  \PROB\left\{  S^{(0)}=\infty \right\} = \frac{1}{\EXP Z}~.
\]
\item[(ii)]
If $q_i>0$ for all $i\in \N$ and $\EXP Z = \infty$, then
\[
  \PROB\left\{  \exists i\le 0: S^{(i)}=\infty \right\} = 0~.
\]
\end{itemize}
\end{theorem}


\begin{proof}
We naturally extend the notation $T^{(i)}$ to positive integers $i\in \N$ 
so that $T^{(i)}$ is the subtree of the random forest rooted at
$i$. Similarly, $S^{(i)}= |T^{(i)}|$ denotes the number of vertices in this
subtree.

In Proposition \ref{prop:finitedegrees} below we show that, regardless
of the distribution of $Z$, there is no vertex of infinite degree,
almost surely.
This implies that the probability that the tree rooted at $0$ is infinite 
equals
\begin{equation}
\label{eq:unionbound}
  \PROB\left\{  S^{(0)}=\infty \right\}
  = \PROB\left\{ \bigcup_{i\in \N} \left\{Z_i=i, S^{(i)}=\infty
    \right\}\right\}
  \le \sum_{i\in \N} q_i \PROB\left\{  S^{(i)}=\infty \right\}~,
\end{equation}
where we used the union bound
and the fact that the events $Z_i=i$ and $S^{(i)}=\infty$ are
independent since the latter only depends on the random variables $Z_{i+1},Z_{i+2},\ldots$.
  Since $\PROB\left\{  S^{(i)}=\infty
\right\} = \PROB\left\{  S^{(0)}=\infty \right\}$ for all $i\ge 1$,
the right-hand side of \eqref{eq:unionbound} equals $ \PROB\left\{
  S^{(0)}=\infty \right\}$, that is, the inequality in
\eqref{eq:unionbound} cannot be strict. This means that the events
$\left\{Z_i=i, S^{(i)}=\infty \right\}$ for $i\in \N$ are disjoint (up
to a zero-measure set).
In particular, almost surely, there are no two maximal infinite paths
meeting at vertex $0$.
By countable additivity, this also implies that
\[
   \PROB\left\{ \text{there exist two infinite paths meeting at any $i\in
       \Z$} \right\} = 0~.
\]
In particular, with probability one, all maximal infinite paths in
the forest are disjoint.

Similarly to \eqref{eq:unionbound}, for all $i\le 0$, 
\[
    \PROB\left\{  S^{(i)}=\infty \right\}
  = 
  \sum_{j: j+i>0} q_j \PROB\left\{  S^{(i+j)}=\infty \right\}
  = \sum_{j: j+i>0} q_j \PROB\left\{  S^{(0)}=\infty \right\}
= p_{1-i} \PROB\left\{  S^{(0)}=\infty \right\}~.
\]
Hence, the expected number of trees in the forest that contain
infinitely many vertices equals
\begin{equation}
\label{eq:expected}
  \EXP \left[ \sum_{i\le 0}\IND_{S^{(i)}=\infty} \right]
= \PROB\left\{  S^{(0)}=\infty \right\} \sum_{i\le 0} p_{1-i} 
    = \PROB\left\{  S^{(0)}=\infty \right\}\EXP Z~.
\end{equation}
If $\EXP Z < \infty$, then by Theorem \ref{thm:finitemean}, the
expectation on the left-hand side equals one.
This implies part (i) of Theorem \ref{thm:extinction}.

It remains to prove part (ii), so assume that $\EXP Z =\infty$.
Suppose first that the left-hand side of \eqref{eq:expected} is finite.
Then we must have $\PROB\left\{  S^{(0)}=\infty \right\}=0$.
But then $\PROB\left\{  S^{(i)}=\infty\right\}=0$ for all $i\ge 0$,
which implies the statement.

Finally, assume that $\EXP \left[ \sum_{i\le 0}\IND_{S^{(i)}=\infty}
\right]= \infty$. This implies that with positive probability, there
are at least two infinite trees in the forest. However, as we show
below, almost surely there is at most one infinite tree in the forest.
Hence, this case is impossible, completing the proof.

It remains to prove that for any $i < j \le 0$,
\[
   \PROB\left\{ S^{(i)}=\infty, S^{(j)}=\infty \right\} = 0~.
\]
For $i < k$, denote by $E_{i,k}=\{k-Z_k=i\}$ the event that $k$ is a vertex in $T^{(i)}$ connected to $i$ 
by an edge (i.e., $k$ is a level $1$ node in the tree $T^{(i)}$). Then by  the union bound,
\begin{eqnarray*}
  \PROB\left\{ S^{(i)}=\infty, S^{(j)}=\infty \right\}
\le 
         \sum_{k,\ell \in \N: k\neq \ell} q_{k-i} q_{j-\ell}
         \PROB\left\{ S^{(k)}=\infty, S^{(\ell)}=\infty | E_{i,k}, E_{j,\ell}  \right\} ~.
\end{eqnarray*}
The key observation is that for all $i<j\le 0$ and $k,\ell \in \N$,
\[
   \PROB\left\{ S^{(k)}=\infty, S^{(\ell)}=\infty | E_{i,k}, E_{j,\ell}  \right\}  = \PROB\left\{ S^{(k)}=\infty, S^{(\ell)}=\infty | E_{0,k}, E_{0,\ell}  \right\}~, 
\]
and therefore
\begin{eqnarray*}
  \PROB\left\{ S^{(i)}=\infty, S^{(j)}=\infty \right\}
& \le &
         \sum_{k,\ell \in \N: k\neq \ell} \frac{q_{k-i} q_{j-\ell}}{q_kq_\ell}  q_kq_\ell
         \PROB\left\{ S^{(k)}=\infty, S^{(\ell)}=\infty | E_{0,k}, E_{0,\ell}  \right\} \\
& = &
         \sum_{k,\ell \in \N: k\neq \ell} \frac{q_{k-i} q_{j-\ell}}{q_kq_\ell} 
         \PROB\left\{ Z_k=k, S^{(k)}=\infty, Z_\ell=\ell, S^{(\ell)}=\infty \right\}~.
\end{eqnarray*}
However, as shown above, each term of the sum on the right-hand side equals zero, which concludes the proof.
\end{proof}

\section{Expected tree sizes}
\label{sec:meantreesize}

In this section we study the expected size of the trees of the random
forest. In particular, we show that in all cases (if $Z$ has full
support), the expected size of
each tree at time $n$ of the forest converges to infinity as $n\to
\infty$. The rate of growth is sublinear if and only if $\EXP Z = \infty$.

Denote the expected
size of the tree rooted at $i$ by $R_n^{(i)} = \EXP S_n^{(i)}$.

\begin{proposition}
\label{prop: treesizes}
{\sc (expected tree sizes.)}
\begin{itemize}
\item[(1)]
  For every $i\in \{0,-1,\ldots \}$, the expected size of the tree
  rooted at $i$
  satisfies
\[
       \lim_{n\to \infty} R_n^{(i)} = \infty~.
\]
Hence, for all distributions of $Z$, we have $\EXP S^{(i)} = \infty$.
\item[(2)]
The sequence $(R_n^{(0)})_{n\ge 0}$ is subadditive, that
is, for all $n,m \ge 0$, $R_{n+m}^{(0)} \le R_n^{(0)}+R_m^{(0)}$
  (where we define $R_n^{(0)}=0$).
\item[(3)]
For every $i\le 0$,
\[
     \lim_{n\to \infty} \frac{R_n^{(i)}}{n} = 0 \quad \text{if and only if} \quad \EXP Z = \infty~.
   \]
\item[(4)]   
If   $\EXP Z < \infty$, then 
\begin{equation}
\label{eq:finitemean}
     \lim_{n\to \infty} \frac{R_n^{(0)}}{n} = \frac{1}{\EXP Z}~.
\end{equation}
\item[(5)]
  Also, for all distributions of $Z$ and for all $i\le 0$
  and $n\in \N$,
\begin{equation}
\label{eq:meantreesize}
   R_n^{(i)} \le 1+ p_{1-i} R_n^{(0)}~.
\end{equation}
\end{itemize}
\end{proposition}  

\begin{proof}
 For $k\in \N$, let $N_k$ denote the number of vertices at path distance $k$
in the tree $T^{(0)}$ rooted at $0$.
Then
\begin{eqnarray*}
  \EXP N_k& = & \sum_{n=1}^\infty \PROB\left\{ \text{vertex} \ n \ \text{connects to $0$ in $k$ steps} \right\}   \\
  & = & \sum_{n=k}^\infty \PROB\left\{ X_1 + \cdots + X_k = n \right\}   \\
           & & \text{(where $X_1,\ldots,X_k$ are i.i.d.\ with the same distribution as $Z$)} \\
  & = & \PROB\left\{ X_1 + \cdots + X_k \ge k \right\} =1~.   \\
\end{eqnarray*}
Hence, $\EXP S^{(0)} = \sum_{k\in \N} \EXP N_k= \infty$, proving 
$(1)$ for the tree rooted at $0$.

In order to relate expected tree sizes rooted at different vertices
$i$, we may consider subtrees rooted at vertices $j \in \{1,2,\ldots \}$.
To this end, let $T^{(j)}_n$ denote the subtree of the forest rooted
at $j$ at time $n$ and let $S^{(j)}_n$ be its size.
Then the size of the tree rooted at $i \in \{0,-1,\ldots \}$
satisfies
\[
    S^{(i)}_n= 1 + \sum_{j=1}^n  \IND_{j-Z_j =i} S^{(j)}_n~.
\]
Noting that  $S^{(j)}_n$ is independent of $Z_j$ and that
$S^{(j)}_n$ has the same distribution as $S^{(0)}_{n-j}$, we obtain
the identity
\[
  R_n^{(i)}= 1+ \sum_{j=1}^n q_{j-i} R_{n-j}^{(0)}~.
\]
  Let $\ell$ be the least positive integer such that $q_{\ell-i}$ is strictly positive.
  The identity above implies that $R_n^{(i)} \ge q_{\ell-i}  R_{n-\ell}^{(0)}$, and
therefore $\lim_{n\to \infty} R_n^{(i)} = \infty$ for all $i$,
proving the first assertion of the theorem.

Using the fact that $R_{n-j}^{(0)} \le R_n^{(0)}$ for all $j\in [n]$,
we obtain \eqref{eq:meantreesize}.

Taking $i=0$ in the equality above, we obtain the following recursion
for the expected size of the tree rooted at $0$, at time $n\in \N$:
\begin{equation}
\label{eq:meantreesizerecursion}
  R_n^{(0)}
  =  1+ \sum_{t=1}^n q_t R_{n-t}^{(0)}~.
\end{equation}
We may use the reursive formula to prove subadditivity of 
the sequence $(R_n^{(0)})_{n\ge 0}$. We proceed by induction.
$R_{n+m}^{(0)} \le R_n^{(0)}+R_m^{(0)}$ holds trivially for all $m\ge
0$ when $n=0$. Let $k\ge 1$. Suppose now that the inequality holds for all $n\le k$
and $m\ge 0$. Then by \eqref{eq:meantreesizerecursion},
\begin{eqnarray*}
   R_{k+m+1}^{(0)} & = & R_{k+1}^{(0)} + \sum_{t=1}^{k+1}q_t \left(R_{k+m+1-t}^{(0)} -
                   R_{k+1-t}^{(0)}\right) + \sum_{t=k+2}^{k+m+1}q_t R_{k+m+1-t}^{(0)}  \\
  & \le &  R_{k+1}^{(0)} + \sum_{t=1}^{k+1}q_t R_m ^{(0)} +
          \sum_{t=k+2}^{k+m+1}q_t R_{k+m+1-t}^{(0)}  \\
            & & \quad \text{(by the induction hypothesis)} \\
  & \le &  R_{k+1}^{(0)} + \sum_{t=1}^{k+1}q_t R_m ^{(0)} +
          \sum_{t=k+2}^{k+m+1}q_t R_m ^{(0)}  \\
             & & \quad \text{(since $R_n ^{(0)}$ is nondecreasing)} \\
  & \le & R_{k+1}^{(0)} + R_m ^{(0)}~,
\end{eqnarray*}
proving $(2)$.

Next we show that if $\EXP Z=\infty$ then $R_n^{(0)}/n \to 0$.
To this end, observe that by \eqref{eq:meantreesizerecursion}, we have
\begin{eqnarray*}
\sum_{i=1}^n R_n^{(0)} & = & n + \sum_{i=1}^n \sum_{t=1}^i q_t
                             R_{i-t}^{(0)} 
                             =  n + \sum_{i=1}^n \sum_{t=0}^{i-1} q_{i-t} R_t^{(0)} \\
                       & = & n + \sum_{t=0}^{n-1} R_t^{(0)} \sum_{i=t+1}^n q_{i-t}
  =  n + \sum_{t=0}^{n-1} R_t^{(0)} (1-p_{n-t+1})~.  
\end{eqnarray*}
Thus,
\begin{eqnarray*}
  n  =  R_n^{(0)} + \sum_{t=0}^{n-1} R_t^{(0)} p_{n-t+1}  
  \ge  R_n^{(0)} + R_{\lfloor n/2 \rfloor}^{(0)} \sum_{t=\lfloor
          n/2 \rfloor}^{n-1} p_{n-t+1}  
    \ge  R_{\lfloor n/2 \rfloor}^{(0)}\sum_{t=2}^{\lfloor n/2 \rfloor} p_t~.
\end{eqnarray*}
Hence,
\[
    \frac{R_{\lfloor n/2 \rfloor}^{(0)}}{n} \le
    \frac{1}{\sum_{t=2}^{\lfloor n/2 \rfloor} p_t} \to 0~.
\]

It remains to prove \eqref{eq:finitemean}.
  (Note that \eqref{eq:finitemean} implies that $R_n^{(i)}/n$ is
  bounded away from zero for all $i\le 0$ and therefore 
  if $R_n^{(i)}/n \to 0$ for some $i\le 0$
then $\EXP Z = \infty$.)
To this end, let
  \[
      f(z) = \sum_{n=0}^\infty R_n^{(0)} z^n
   \] 
be the generating function of the sequence $\{R_n^{(0)}\}_{n\ge 0}$, where $z$
is a complex variable.
Using the recursion \eqref{eq:meantreesizerecursion}, we see that
\begin{eqnarray*}
    f(z) & = & \sum_{n=0}^\infty z^n + \sum_{n=1}^\infty \sum_{t=1}^n
               q_t R_{n-t}^{(0)} z^n \\
         & = &
               \sum_{n=0}^\infty z^n + \sum_{t=1}^\infty q_t z^t \sum_{n=t}^\infty
               R_{n-t}^{(0)} z^{n-t}  \\
       & = & \frac{1}{1-z}+ Q(z) f(z)~,
\end{eqnarray*}
where $Q(z) = \sum_{n=1}^\infty q_nz^n$ is the generating function of
the sequence $\{q_n\}_{n\ge 1}$.
Thus, we have
\[
     f(z) = \frac{1}{(1-z)(1-Q(z))}~.
\]
Recall that we assume here that $\EXP Z < \infty$. Since $1-Q(z) \sim (1-z)  \sum_{n=1}^\infty n q_n=(1-z) \EXP Z$ when
$z\to 1$, we have
\[
    f(z) \sim \frac{1}{(1-z)^2\EXP Z} \quad \text{when $z\to 1$}~.
  \]
 \eqref{eq:finitemean} now follows from Corollary VI.1 of Flajolet and
 Sedgewick \cite{FlSe09}.
\end{proof}

\begin{remark}
{\sc (profile of the $0$-tree.)}
Note that we proved in passing that, regardless of the distribution,
for all $k \in \N$, the number $N_k$ of vertices in the
tree $T^{(0)}$ that are at path distance $k$ from the root satisfies $\EXP N_k=1$.
The
sequence $\{N_k\}_{k=1}^\infty$ is often called the \emph{profile} of the tree
$T^{(0)}$.
\end{remark}

\begin{remark}
{\sc (expected vs.\ actual size.)}
While Proposition \ref{prop: treesizes} summarizes the properties of the \emph{expected}
tree sizes $R_n^{(i)} = \EXP S_n^{(i)}$, it is worth emphasizing that the random variables
$S_n^{(i)}$ behave very differently. For example, when $\EXP Z = \infty$, then we know from
Theorem \ref{thm:extinction} that for each $i\le 0$, $S^{(i)}= \limsup_{n\to\infty} S_n^{(i)} < \infty$,
while, by Proposition \ref{prop: treesizes}, $\EXP S_n^{(i)} \to
\infty$. Also note that, for all distributions of $Z$, 
for each $i\le 0$, $\PROB\{S^{(i)}=1\}$ is strictly positive and therefore $S_n^{(i)}$ does
not concentrate.
\end{remark}

\section{The number of trees of the forest}
\label{sec:numberoftrees}

In this section we study the number
$M_n= \sum_{i \in  \{0,-1,-2,\ldots\}} \IND_{S^{(i)}_n >1}$
of trees in the random forest that have at least one vertex $t\in [n]$
attached to the root $i$. In the motivating topic recommendation problem,
this random variable describes the number of topics that are
recommended by time $n$.

We show that the expected number of trees $\EXP M_n$ goes to infinity as $n\to \infty$  if and only
if $\EXP Z = \infty$. Moreover, $M_n \sim m_n$ in probability, where
$m_n= \EXP\left[ \min(Z,n)\right]$.

Note that it follows from Theorem \ref{thm:extinction} that if $\EXP Z = \infty$, then
$M_n \to \infty$ almost surely.

In order to understand the behavior of $M_n$, we first study the
random variable
\[
    O_n = \sum_{t=1}^n\IND_{Z_t \ge t}~.
\]
Note that when $Z_t\ge t$, then   vertex $t$ connects directly to the
root $t-Z_t \le 0$. 
Hence, $O_n$ is the number of vertices in the forest at depth
$1$ (i.e., at graph distance $1$ from the root of the tree containing the vertex). 
Equivalently, $O_n= \sum_{i \in \{0,-1,\ldots\}} D_n^{(i)}$ is the sum
of the degrees of the roots of all trees in the forest at time $n$.

\begin{proposition}
\label{prop:numberoftrees}
{\sc (number of trees.)}
The random variables $M_n$ and $O_n$ satisfy the following:
\begin{itemize}
\item[(i)] $\EXP O_n = m_n$~;
\item[(ii)]  If $\EXP Z = \infty$, then $(O_n-m_n)/\sqrt{m_n}$~;
  converges, in distribution, to a standard normal random variable.
\item[(iii)] $M_n \le O_n$ and if $\EXP Z = \infty$, then
  $\EXP[O_n-M_n] = o(m_n)$. 
\item[(iv)]  If $\EXP Z = \infty$, then $M_n/m_n \to 1$ in probability.
\item[(v)] For all $x>0$,
\[
    \PROB\{ M_n > \EXP M_n+x \}  \le\left(\frac{\EXP M_n}{\EXP M_n +x}\right)^{\EXP M_n+x} e^{-x}
\]
and
\[
     \PROB\{ M_n < \EXP M_n- x \}  \le e^{-x^2/(2\EXP M_n)}~.
\]
\end{itemize}
\end{proposition}

\begin{proof}
Note that $O_n$ is a sum of independent Bernoulli random variables
and 
\[
    \EXP O_n = \sum_{t=1}^n p_t = m_n~.
\]
To prove (ii), we may use Lyapunov's central limit theorem.
Indeed,
\[
    \var(O_n)= \sum_{t=1}^n p_t(1-p_t)= m_n - \sum_{t=1}^n p_t^2~.
\]
If $\EXP Z = \infty$, then $\sum_{t=1}^n p_t^2 = o(m_n)$. (This simply
follows from the fact that $p_t \to 0$.)
In order to use Lyapunov's central limit theorem, it suffices that
\[
\frac{\sum_{t=1}^n \EXP |\IND_{Z_t \ge t} - p_t|^3}{\var(O_n)^{3/2}} \to 0~.
\]
This follows from
\[
  \sum_{t=1}^n \EXP |\IND_{Z_t \ge t} - p_t|^3
  = \sum_{t=1}^n  p_t (1 - p_t) \left((1 - p_t)^2+ p_t^2\right)
   \le \var(O_n)~.
\]
In order to prove (iii), observe that for each $i\in \{0,-1,-2,\ldots \}$,
\[
  \EXP \IND_{S^{(i)}_n >1} = 1 - \PROB\left\{S^{(i)}_n =1 \right\}
   = 1 - \prod_{t=1}^n (1- q_{t-i})~,
\]
and therefore
\begin{eqnarray*}
    \EXP M_n & \ge & \sum_{i \le 0} \left( 1 - e^{-\sum_{t=1}^n
                     q_{t-i}}\right)  \\
             & = & \sum_{i \le 0}\left( 1 - e^{-(p_{1-i}-p_{n+1-i})}\right) \\
  & \ge & \sum_{i \le 0} (p_{1-i}-p_{n+1-i})  - \frac{1}{2} \sum_{i \le
          0} (p_{1-i}-p_{n+1-i})^2 \\
             & & \text{(using $e^{-x} \le 1-x+x^2/2$ for $x\ge 0$)}  \\
  & = & \sum_{t=1}^n p_t   - \frac{1}{2} \sum_{i \le
          0} (p_{1-i}-p_{n+1-i})^2 = m_n (1-o(1))~,
\end{eqnarray*}
where the last assertion follows from the fact that $p_{1-i}-p_{n+1-i}
\to 0$ as $i\to -\infty$
and that $\sum_{t=1}^n p_t\to \infty$ when $\EXP Z =\infty$.
Part (iv) simply follows from (ii), (iii), and Markov's inequality. Indeed,
$M_n\le O_n$ and for every $\epsilon >0$,
\[
  \PROB\left\{  O_n -M_n > \epsilon m_n \right\} \le
  \frac{\EXP[O_n-M_n]}{\epsilon m_n} = o(1)~.
\]
The exponential inequalities of (v) follow from the fact that the
collection of indicator random variables
$\left\{\IND_{Z_n=n-i}\right\}_{n\ge 1,  i\le 0}$ is negatively associated
(Dubhashi and Ranjan \cite[Proposition 11]{DuRa98}).
This implies that the collection of indicators
\[
  \left\{ \IND_{S^{(i)}_n >1}\right\}_{i\le 0} =
    \left\{ \IND_{\sum_{t=1}^n \IND_{Z_t=t-i} >0} \right\}_{i\le 0}
\]
 is also negatively associated (\cite[Proposition 7]{DuRa98}).
 Hence, by \cite[Proposition 5]{DuRa98}, the tail probabilities of
  the sum $M_n = \sum_{i\le 0} \IND_{S^{(i)}_n >1}$ satisfy 
  the Chernoff bounds for the corresponding sum of independent random
  variables.
 The inequalities of (v) are two well-known examples of the Chernoff
 bound (see, e.g., \cite{BoLuMa13}).
\end{proof}


\section{Number of leaves}
\label{sec:leaves}

Let $L_n$ denote the number of leaves of the tree rooted at $0$, at
time $n$. That is, $L_n$ is the number of vertices $t\in [n]$ such
that $C_t=0$ and no vertex $s\in \{t+1,\ldots,n\}$ is attached to it.
Recall that $r_n = \PROB\{C_n=0\}$ is the probability that vertex $n$
belongs to the tree $T_n^{(0)}$ rooted at $0$ and $R_n^{(0)}= \sum_{t=1}^n r_t$
is the expected size of the tree $T_n^{(0)}$.
The following proposition shows that the expected number of leaves is 
proportional to the expected number of vertices in the tree.

\begin{proposition}
\label{prop:leaves}
{\sc (number of leaves.)}
Denote $q_{max}= \max_n q_n$.
If $\EXP Z = \infty$, then there exists a constant $c \in [e^{-1/(1-q_{max})},
e^{-1}]$ such that
\[
  \lim_{n\to \infty}   \frac{\EXP L_n}{R_n^{(0)}} = c~.
\]
\end{proposition}

\begin{proof}
Let $t\in [n]$.   
Since the event $\{C_t=0\}$ is independent of the event that 
no vertex $s\in \{t+1,\ldots,n\}$ is attached to $t$, we may write
\begin{eqnarray*}
  \EXP L_n & = & \sum_{t=1}^n r_t  \prod_{s=t+1}^n (1-q_{s-t})  \\
           & = & \sum_{t=1}^n r_t  \prod_{s=1}^{n-t} (1-q_s)  \\
           & = & \sum_{t=1}^n r_t  e^{ \sum_{s=1}^{n-t} \log(1-q_s)}~.
\end{eqnarray*}
The sequence $\left(\sum_{s=1}^n \log(1-q_s)\right)_{n\ge 1}$ is monotone decreasing and,
using that $\log(1-x) \ge -x/(1-x)$ for $x \ge 0$, 
\[
  \sum_{s=1}^n\log(1-q_s) \ge \sum_{s=1}^n \frac{-q_s}{1-q_s}
  \ge \sum_{s=1}^n \frac{-q_s}{1-q_s} \ge \frac{-1}{1-q_{max}}~.
\]
Thus, there exists $c \ge e^{-1/(1-q_{max})}$ such that for all
$\epsilon >0$, there exists $n_0\in \N$ such
that $|e^{ \sum_{s=1}^n \log(1-q_s)} -c| < \epsilon$ whenever
$n>n_0$.
But then
\begin{eqnarray*}
\left| \EXP L_n- c R_n^{(0)} \right| & = &
\left| \sum_{t=1}^n r_t \left( e^{ \sum_{s=1}^{n-t} \log(1-q_s)}
                                            -c\right) \right|   \\
                                      & \le &
                                              n_0 + \epsilon \sum_{t=1}^{n-n_0} r_t                  \\
  & \le & \epsilon R_n^{(0)} + O(1)~.
\end{eqnarray*}
To see why $c \le 1/e$, note that 
\[
  \sum_{s=1}^n\log(1-q_s) \le - \sum_{s=1}^n q_s \to -1~.
\]
\end{proof}

\begin{remark}
{\sc (expected vs.\ actual random number of leaves.)}
Just like the size of the trees, the number of leaves $L_n$ is
not concentrated around its expectation. Indeed, when $\EXP Z=\infty$,
$L_n \le S^{(0)}$ is an almost surely bounded sequence of random variables, whereas
$\EXP L_n \to \infty$ by Propositions \ref{prop: treesizes} and \ref{prop:leaves}.
\end{remark}

\begin{remark}
{\sc (number of leaves when $\EXP Z < \infty$)}
Proposition \ref{prop: treesizes} is only concerned with the case $\EXP Z=\infty$.
When $Z$ has finite expectation, then the number of leaves of the tree rooted at $0$
depends on whether the tree survives or not. Recall that the events
$S^{(0)}<\infty$
and $S^{(0)} =\infty$
both have positive probability.
It is easy to see that, conditioned on the event $S^{(0)}=\infty$, the ratio $L_n/S_n^{(0)}$ almost surely converges to
$\PROB\{S^{(0)}=1\}$. On the other hand, conditioned on the event $S^{(0)}<\infty$,
$L_n/S_n^{(0)}$ converges to a nontrivial random variable taking values in $[0,1]$.
\end{remark}

\section{Degrees}
\label{sec:degrees}

The \emph{outdegree} $D_n^{(i)}$ of a vertex $i\in \Z$ is the number of vertices
attached to it at time $n\in \N$, that is,
\[
    D_n^{(i)} = \sum_{t=\max(1,i+1)}^n \IND_{Z_t=t-i}~.
\]
We also write 
\[
    D^{(i)} = \sum_{t=\max(1,i+1)}^\infty \IND_{Z_t=t-i}
  \]
for the degree of vertex $i$ in the random forest at the end of the
attachment process.
Note that for all root vertices $i\le 0$, $\EXP D_n^{(i)} =
\sum_{t=1}^n q_{t-i}= p_{1-i}-p_{n-i+1}$  and $\EXP D^{(i)} =p_{1-i}$,
while for all other vertices $i\ge 1$,
$\EXP D_n^{(i)} = 1-p_{n-i+1}$  and $\EXP D^{(i)}=1$.

First we show that the degrees among all root vertices is a
tight sequence of random variables under general conditions, with the possible
exception of some extremely heavy-tailed distributions. 

\begin{proposition}
\label{prop:rootdegrees}
{\sc (maximum root degree.)}
If the distribution of $Z$ is such that there exists
$\lambda >0$ such that $\sum_{n=1}^\infty p_n^\lambda < \infty$, then
the root degrees $\{ D^{(i)}\}_{i <0}$ form a tight sequence of random
variables. In particular, for all $x\ge \lambda$, we have
\[
    \PROB\left\{ \max_{i <0} D^{(i)}> x \right\} \le
    \left(\frac{e}{x}\right)^x \sum_{n=1}^\infty p_n^x~.
\]
\end{proposition}

As an example, consider a distribution with polynomially decaying tail
such that $q_n = \Theta(n^{-1-\alpha})$ for some $\alpha>0$.
Then $p_n = \Theta(n^{-\alpha})$, and then for any $x>1/\alpha$, we have
$\sum_{n=1}^\infty p_n^x < \infty$. However, if $q_n$ decreases much
slower, for example, if $q_n \sim 1/ (n\log^2 n)$, then the
proposition does not guarantee tightness of the root degrees.

\begin{proof}We have
\begin{eqnarray*}
  \PROB\left\{\max_{i <0} D^{(i)} > x \right\}
  & \le &
     \sum_{i\le 0} \PROB\left\{   \sum_{t=1}^\infty \IND_{Z_t=t-i}  >
          x \right\} \\
  & \le &
     \sum_{i\le 0}  \exp\left( x - p_{1-i} - x \log\frac{x}{p_{1-i}}
          \right) \\
  & & \text{(by the Chernoff bound)}  \\
   & \le &
       \sum_{i\le 0}  \exp\left( x  - x \log\frac{x}{p_{1-i}}
           \right) \\
   & = &
          \left(\frac{e}{x}\right)^x \sum_{n=1}^\infty p_n^x~,
\end{eqnarray*}  
which proves the claim.
\end{proof}

Next we show that the maximum degree of any vertex $t\in [n]$ grows at
most as the maximum of independent Poisson$(1)$ random variables
that is well known (and easily seen) to grow as $\log
n/\log\log n$.

\begin{proposition}
\label{prop:degrees}
{\sc (maximum degree.)}
For every $\epsilon >0$, with probability tending to $1$, 
\[
    \max_{t \in [n]} D^{(t)}\le (1+\epsilon) \frac{\log n}{\log\log n}~.
  \]
\end{proposition}

\begin{proof}
The proof once again follows from a simple application of the Chernoff
bound for sums of independent Bernoulli random variables: for any $x>0$,
\begin{equation}
\label{eq:maxdegree}  
  \PROB\left\{ \max_{t \in [n]} D^{(t)}> x \right\}
  \le n e^{x-1 - x\log x}~,
\end{equation}
which converges to $0$ if $x= (1+\epsilon) \frac{\log n}{\log\log n}$
for any fixed $\epsilon >0$.
\end{proof}

\begin{proposition}
\label{prop:finitedegrees}
{\sc (all degrees are finite.)}
With probability $1$, $D^{(t)}<\infty$ for all $t\in \Z$.
\end{proposition}

\begin{proof}
Bounding as in the proof of Proposition
\ref{prop:rootdegrees}, we see that, for every $n \in \N$
  and $x>0$,
\[
  \PROB\left\{\max_{i \in \{0,-1,\ldots,-n\}} D^{(i)} > x \right\}
   \le 
        n\exp\left( - x \log\frac{x}{e}           \right)~.
\]
Hence, by the Borel-Cantelli lemma,
 $\max_{i \in \{0,-1,\ldots,-n\}} D^{(t)} \le 3 \log n$ for all
but finitely many values of $n$, almost surely. This implies that,
almost surely, $D^{(t)}<\infty$
for all $t\le 0$.

Similarly, by taking (say) $x=3\log n$ in \eqref{eq:maxdegree}, it follows from the Borel-Cantelli lemma
that, almost surely, $\max_{t \in [n]} D^{(t)} \le 3 \log n$ for all
but finitely many values of $n$. This implies that $D^{(t)}<\infty$
for all $t\in \N$,
with probability one.
\end{proof}

\begin{remark}
{\sc (asymptotic distribution of the out-degree.)}
As argued above, the asymptotic degree of vertex $i$ may be
represented as a sum of independent Bernoulli random variables
\[
    D^{(i)}= \sum_{t=\max(1,i+1)}^\infty \IND_{Z_t=t-i} = \sum_{t=\max(1,i+1)}^\infty \text{Ber}(q_{t-i})~.
\]
For example, for all $i\ge 0$, the $D^{(i)}$ are discrete random variables with the same distribution,
satisfying $\EXP D^{(i)}=1$ and $\var(D^{(i)})=1- \sum_{t=1}^\infty q_t^2$.
\end{remark}

\section{The height of the random forest}
\label{sec:height}

  In this section we study the expected \emph{height} of the random forest.
  The height
$H_n$ of the forest, at time $n\in \N$, is the length of the longest
path of any vertex $t\in [n]$ to the root of its tree.
In Proposition \ref{prop:height} we derive an upper bound for $\EXP H_n$.
The upper bound
implies that the expected height is sublinear whenever
$q_n = \Theta(n^{-1-\alpha})$ for some $\alpha \in (0,1)$.

In Proposition \ref{prop:height-lower} we show that the expected height $H^{(0)}_n$
of the tree rooted at vertex $0$ goes to infinity, regardless of the distribution of $Z$.
Of course, this implies that $\EXP H_n \to \infty$. As a corollary, we
also show that for all distributions, $H_n \to \infty$ almost
surely. This is to be contrasted with the fact that when $\EXP Z =
\infty$, $H^{(0)}_n$ is almost surely bounded (just like the height of
any tree in the forest).

\begin{proposition}
\label{prop:height}
{\sc (upper bound for the expected height of the forest.)}
For all distributions of $Z$, we have
\[
     \EXP H_n\le \frac{2+\log n}{p_n}~.
\]
\end{proposition}

\begin{proof}
The path length of a vertex $n$ to the root of its tree exceeds $k$ if
and only if
\[
  \underbrace{Z_n+Z_{n-Z_n} + Z_{n-Z_n-Z_{n-Z_n}} + \cdots
  }_{\text{$k$ times}} < n
\]
Thus, if $X_1,\ldots,X_k$ are i.i.d.\ with the same distribution as
$Z$,
\begin{eqnarray*}
  \PROB\left\{ H_n >k \right\}
  & \le &
          \sum_{t=1}^n \PROB\left\{ X_1+ \cdots + X_k < t \right\}  \\
  & \le &
          \sum_{t=1}^n  \PROB\left\{ Z < t \right\}^k  \\
  & = &
          \sum_{t=1}^n  \left( 1 - \PROB\left\{ Z \ge t \right\}\right)^k  \\
  & \le &
          \sum_{t=1}^n e^{-kp_t}~.
\end{eqnarray*}
This implies that
\begin{eqnarray*}
  \EXP H_n & = &
                 \sum_{k=0}^\infty \PROB\left\{ H_n >k \right\}  \\
           & \le &
              \sum_{k=0}^\infty  \min\left(1, \sum_{t=1}^n e^{-kp_t} \right)                   \\
           & \le &
                   \sum_{k=0}^\infty  \min\left(1, n e^{-kp_n} \right)
                   \\
           & \le & \frac{\log n}{p_n}+\sum_{k= \left\lceil\frac{\log
                   n}{p_n}\right\rceil}^\infty  n e^{-kp_n}
                    \\
           & \le & \frac{\log n}{p_n}+ \frac{1}{1-e^{-p_n}}~.
\end{eqnarray*}
Using the fact that $e^{-x}\ge 1-x/2$ for $x\in [0,1]$, we obtain the
announced inequality.
\end{proof}

\begin{proposition}
\label{prop:height-lower}
{\sc (lower bound for the expected height of the forest.)}
For all distributions of $Z$, the expected height of the tree rooted at vertex $0$
satisfies
\[
    \lim_{n \to \infty} \EXP H^{(0)}_n= \infty~.
\]
\end{proposition}

\begin{proof}
  Since $H^{(0)}_n$ is an increasing sequence of random variables, we may define
  $H^{(0)}=\lim_{n\to \infty} H^{(0)}_n$ (that may be infinite). By the monotone convergence
  theorem, it suffices to prove that $\EXP H^{(0)}=\infty$, or equivalently, that
  \begin{equation}
\label{eq:height}
   \sum_{k=1}^\infty \PROB\left\{ H^{(0)} \ge k \right\}= \infty~.
 \end{equation}
 Denote by $t\to_k 0$ the event that vertex $t>0$ is connected to vertex $0$ via a path
 of length $k$, that is,
 \[
  \underbrace{Z_t+Z_{t-Z_t} + Z_{t-Z_t-Z_{t-Z_t}} + \cdots
  }_{\text{$k$ terms}} =t
\]
and define $r_{t,k}=\PROB\left\{t\to_k 0\right\}$.

Introducing the random variable
\[
   N_k= \sum_{t=k}^\infty \sum_{\ell=k}^{2k-1}\IND_{t\to_\ell 0}~,
 \]
 note that
\[
  \PROB\left\{ H^{(0)} \ge k \right\} \ge \PROB\left\{ H^{(0)} \in [k,2k-1] \right\}
  = \PROB\{ N_k \ge 1\}~. 
\]
In order to derive a lower bound for $\PROB\{ N_k \ge 1\}$, note first that
\[
  \EXP N_k = \sum_{\ell=k}^{2k-1} \PROB\{X_1+\cdots + X_\ell \ge k\} = k
\]
where $X_1,\ldots,X_k$ are i.i.d.\ with the same distribution as
$Z$.
By the Paley-Zygmund inequality,
\[
  \PROB\{ N_k \ge 1\} \ge \left(1-\frac{1}{k}\right)^2 \frac{\left(\EXP N_k\right)^2}{\left(\EXP N_k^2\right)}~.
\]
In the argument below we show that $\EXP N_k^2\le 4k^3$. Substituting
into the inequality above, we obtain
\[
\PROB\left\{ H^{(0)} \ge k \right\} \ge  \PROB\{ N_k \ge 1\} \ge \left(1-\frac{1}{k}\right)^2 \frac{1}{4k}~,
\]
concluding the proof of $\EXP H^{(0)}=\infty$.

Hence, it remains to derive the announced upper bound for the second moment of $N_k$.
First note that for any $t,t' \ge k$ and $\ell,\ell' \in
\{k,k+1,\ldots,2k-1\}$
\[
    \PROB\{t\to_\ell 0, t'\to_{\ell'} 0\} \le
    \sum_{m=0}^{\min(\ell,\ell')} \sum_{s=0}^{\min(t,t')} r_{s,m}
    r_{t-s,\ell-m} r_{t'-s,\ell'-m}~.
\]
Then
\begin{eqnarray*}
  \EXP N_k^2 & = & \sum_{t=k}^\infty \sum_{t'=k}^\infty
                   \sum_{\ell=k}^{2k-1}\sum_{\ell'=k}^{2k-1}
                   \sum_{m=0}^{\min(\ell,\ell')} \sum_{s=0}^{\min(t,t')} r_{s,m}
                   r_{t-s,\ell-m} r_{t'-s,\ell'-m} \\
             & \le &
                     \sum_{t=k}^\infty 
                   \sum_{\ell=k}^{2k-1}\sum_{\ell'=k}^{2k-1}
                   \sum_{m=0}^{\min(\ell,\ell')} \sum_{s=0}^{\min(t,t')} r_{s,m}
                     r_{t-s,\ell-m} 
  \quad \text{(using $\sum_{t'=k}^\infty r_{t'-s,\ell'-m} \le 1$)}
  \\
             & \le &
     2k   \sum_{t=k}^\infty 
                   \sum_{\ell=k}^{2k-1}\sum_{\ell'=k}^{2k-1}
                     r_{t,\ell}   \quad \text{(since for each $m$, $\sum_{s=1}^{\min(t,t')} r_{s,m}
                     r_{t-s,\ell-m} \le r_{t,\ell}$)}  \\
           & \le &
        4k^2 \sum_{t=k}^\infty 
                   \sum_{\ell=k}^{2k-1}
                     r_{t,\ell} \\
 &=& 4k^2 \EXP N_k = 4k^3~,
\end{eqnarray*}
as desired.
\end{proof}

\begin{proposition}
\label{prop:height-as}
{\sc (almost sure lower bound for the height of the forest.)}
For all distributions of $Z$, $\lim_{n\to \infty} H_n = \infty$ almost surely.
\end{proposition}

\begin{proof}
 For $\EXP Z < \infty$, the statement follows from Theorem \ref{thm:finitemean}
and Proposition \ref{prop:finitedegrees}
  so we may assume that
$Z$ has infinite expectation.

Since by Proposition \ref{prop:height-lower} the
expected height of the tree rooted at $0$ has infinite expectation, it
follows that the distribution of $H^{(0)}$ has unbounded support. 

Since by Theorem \ref{thm:extinction} the tree $T^{(0)}$ rooted at $0$ becomes
extinct almost surely, the random variable $Y_1$ denoting the index of
the last vertex that belongs to $T^{(0)}$ is almost surely finite. Let
$A_1=H^{(0)}$ denote the height of the $0$-tree. Now we may define
$Y_2$ such that $Y_1+1+Y_2$ is the last vertex that belongs to the
tree $T^{(Y_1+1)}$, and let $A_2$ denote the height of this tree.
By continuing recursively, we obtain a sequence $A_1,A_2,\ldots$ of
i.i.d.\ random variables distributed as $H^{(0)}$. Moreover,
$\lim_{n\to \infty} H_n \ge \sup_i A_i$, proving the statement.
\end{proof}

  \noindent
  {\bf Acknowledgements.} We are grateful to an anonymous referee for suggestions that
lead to a much improved version of the paper.


\end{document}